\documentclass{article}
\usepackage[utf8]{inputenc}
\usepackage[margin=1.5in]{geometry}
\usepackage{todonotes}
\usepackage{silence}
\usepackage{hyperref}
\usepackage{caption}
\usepackage{mathdots}
\usepackage{nicefrac}
\usepackage{amsmath,enumerate}
 \usepackage{amsthm}
\usepackage{amsfonts}
\usepackage{amssymb}
\usepackage{mathrsfs}
\usepackage{pgfplots}
\usepackage{mathtools}
\usepackage{comment}
\usepackage{lmodern}
\usepackage{subfig}
\usepackage{caption}
\usepackage{accents}
\usepackage[style= numeric-comp,hyperref=true, doi=false,url=false,
            isbn=false,
            firstinits=true, sorting = none, 
            block=none, backend=bibtex,maxnames=99]{biblatex}
\renewbibmacro{in:}{} 
\bibliography{feedback}

\usepackage{tikz}
\usetikzlibrary{shapes,arrows,graphs,graphs.standard,shapes.misc}
\usetikzlibrary{matrix,decorations.pathreplacing, calc, positioning,fit}
\usetikzlibrary{shapes.geometric}

\newtheorem{Theorem}{Theorem}[section]
\newtheorem{Definition}{Definition}[section]

\newtheorem{Proposition}[Theorem]{Proposition}
\newtheorem{Lemma}[Theorem]{Lemma}
\newtheorem{Corollary}[Theorem]{Corollary}
\newtheorem{Assumption}{Assumption}[section]
\newtheorem{Remark}{Remark}[section]

\newcommand{\xc}[1]{\vspace{.1cm}
\noindent{\em #1}}

\newcommand{\re}{\operatorname{re}}

\newcommand{\R}{\mathbb{R}}

\newcommand{\N}{\mathbb{N}}

\newcommand{\C}{\mathbb{C}}
\newcommand{\rd}{\mathrm{d}}

\newcommand{\bfo}{\mathbf{1}}

\newcommand{\Z}{\mathbb{Z}}

\DeclareMathOperator{\rk}{rk}

\newcommand{\rP}{\mathrm{P}}

\newcommand{\Diag}{\operatorname{Diag}}

\title{Pole Placement and Feedback Stabilization for \\ Discrete Linear Ensemble Systems}
\begin{document}
\author{Xudong Chen\footnote{X.~Chen is with the Department of Electrical and Systems Engineering, Washington University in St. Louis. Email: \texttt{cxudong@wustl.edu}.}
}

\date{}
\maketitle

\begin{abstract}
We consider discrete ensembles of linear, scalar control systems with single-inputs. Assuming that all the individual systems are unstable, we investigate whether there exist linear feedback control laws that can asymptotically stabilize the ensemble system. 
We provide necessary/sufficient conditions for feasibility of pole placement in the left half plane and for feedback stabilizability of the ensemble systems.        
\end{abstract}

\section{Introduction} 
\subsection{Problem formulation}
Let $X$ be a Banach sequence space in $\C$. In this paper, $X$ will be either $\ell^p$, for $1\leq p\leq \infty$, or, $c$ the space of convergence sequences, or, $c_0$ the space of null sequences, i.e., sequences that converge to $0$.      
We consider discrete ensembles of linear, scalar control systems with single-inputs: 
\begin{equation}\label{eq:systeminfinite}
    \dot x_n(t) = a_n x_n(t) + b_n u(t), \quad \mbox{for } n \in \N, 
\end{equation}
where $a:= (a_n)\in \ell^\infty$, $b:= (b_n) \in X$, $x(t):= (x_n(t))\in X$, and $u(t)$ is complex valued and locally integrable. 
For the case $X = c$, we further require that $a\in c$.  
Let $A: X \to X$ be the diagonal operator given by $A: (x_n) \mapsto (a_n x_n)$. Then, one can  re-write  system~\eqref{eq:systeminfinite} simply as follows:
\begin{equation}\label{eq:openloopsystem}
\dot x(t) = A x(t) + b u(t).
\end{equation} 
Let $X^*$ be the dual space of $X$. 
A {\em linear feedback control law} takes the form  
$u(t) = k x(t)$, for some $k\in X^*$.  
Correspondingly, the feedback system is given by 
\begin{equation}\label{eq:feedbacksystem}
\dot x(t) = T_k \, x(t), \quad \mbox{where } T_k:= A + bk.  
\end{equation} 
We have the following definition:

\begin{Definition}
    System~\eqref{eq:feedbacksystem} is {\bf stable} if there exists a constant $C > 0$ such that for any initial condition $x(0)$, we have $\|x(t)\|_X \leq C \|x(0)\|_X$. System~\eqref{eq:feedbacksystem} is {\bf asymptotically stable} if it is stable and, moreover, for any initial condition $x(0)\in X$, we have $\lim_{t\to\infty} x(t) = 0$. 
\end{Definition}

\begin{Remark}\label{rmk:strongstabilitysemigroup}
\normalfont
The operator $T_k$ generates the uniformly continuous semigroup $(\exp(T_k t))$, for $t \geq 0$. The above definition of system stability is closely related to the stability notions for semigroups (see, e.g.,~\cite[Chapter V]{engel2000one}). Specifically, the semigroup $(\exp(T_k t))_{t \geq 0}$ is said to be {\em strongly stable} if for any $x\in X$, we have $\lim_{t\to\infty } \|\exp(T_kt)x\|_X = 0$. Thus, system~\eqref{eq:feedbacksystem} is asymptotically stable if and only if the associated semigroup $(\exp(T_k t))_{t\geq 0}$ is bounded and strongly stable.     \hfill{\qed}
\end{Remark}

Let $\Sigma(T_k)$ be the spectrum of $T_k$. It should be clear that 
a necessary condition for~\eqref{eq:feedbacksystem} to be (asymptotically) stable is that 
\begin{equation}\label{eq:necessarycondition}
\Sigma(T_k) \subseteq H:= \{z \in \C \mid \re(z) \leq 0\}.
\end{equation}
The two main questions we address in this paper are the following: 
\begin{description}
\item[\it 1. Pole placement:] {\em Is there a $k\in X^*$ such that~\eqref{eq:necessarycondition} can be satisfied?} 
\item[\it 2. Feedback stabilization:] {\em If the answer to the first question is affirmative, then is there a $k\in X^*$ such that~\eqref{eq:feedbacksystem} is asymptotically stable?} 
\end{description}

The ensemble system~\eqref{eq:systeminfinite} being {\em discrete} 
is essentially necessary for the above questions to have positive solutions. To wit, consider, e.g., the following {\em continuum} linear ensemble system $\dot x(t, a) = a x(t, a) + u(t)$, for $a\in [0,1]$. In this case, $A: \mathrm{L}^p ([0,1],\R)\to \mathrm{L}^p ([0,1],\R)$ is the multiplication operator $A: f(x) \mapsto  a f(a)$, with $\Sigma(A) = [0,1]$, and $b$ is the constant function $\bfo$. Since there is no compact operator $K$ such that $\Sigma(A + K) \subseteq H$ (see, e.g.,~\cite[\S 5.2]{curtain2012introduction}), there is no linear feedback control law that can render the closed-loop system stable. This observation can be formalized and extended, with mild efforts, to more general cases.   
For other relevant works about continuum linear ensemble systems, we refer the reader to~\cite{helmke2014uniform,li2015ensemble,dirr2021uniform,chen2021sparse,chen2023controllability} and references therein. We further mention~\cite{li2006control,chen2019structure,chen2020ensemble} for studies of continuum nonlinear ensemble systems. 

For system~\eqref{eq:systeminfinite}, if there are only finitely many  $a_n$'s with positive real parts, say $\re(a_n) > 0$, for $n = 1,\ldots, N$, and $\re(a_n) \leq 0$ for all $n > N$, then the problem of feedback stabilization is reduced to the finite-dimensional case. Specifically, let $A':= \Diag(a_1,\ldots, a_N)$ and $b' := (b_1, \ldots, b_N)$. The pair $(A',b')$ is controllable if and only if 
the $a_n$'s, for $1\leq n \leq N$, are pairwise distinct and the $b_n$'s, for $1\leq n \leq N$, are nonzero. 
On one hand, if $(A',b')$ is not controllable, then one can show that for any $k\in X^*$, $\Sigma(T_k)$ contains $a_n$ for some $n = 1,\ldots, N$ and hence,~\eqref{eq:necessarycondition} cannot be satisfied. On the other hand, if $(A',b')$ is controllable, then it is well known that there exists a row vector $k'\in \R^N$ such that the eigenvalues of $(A' + b'k')$ can be placed in the interior of $H$. 
Then, if we let $k\in X^*$ be such that
$k: x \mapsto \sum_{n = 1}^N k'_n x_n$, then one can show that the feedback system~\eqref{eq:feedbacksystem} is stable (and is asymptotically stable if $a_n < 0$ for all $n > N$). 

In this paper, we take the following assumption: 

\begin{Assumption}\label{asmp:positivean}
All the $a_n$'s are positive real numbers. 
\end{Assumption}

Although the ensemble system~\eqref{eq:systeminfinite}, with Assumption~\ref{asmp:positivean}, is simple, the questions of pole placement and of feedback stabilization are nontrivial. The main results are formulated in Section~\S\ref{sec:mainresult}, where we provide partial solutions to these two questions.

\subsection{Literature review} 
Although feedback stabilization is a central topic in control theory, the literature is very sparse for ensemble systems. 
For relevant works, we mention~\cite{chittaro2018asymptotic} in which the authors addressed the problem of feedback stabilizing a discrete ensemble of Bloch equations, with the target state being the south pole $v := (0,0,-1)$ for all individual systems. 
They have shown that for any given error tolerance $\epsilon > 0$, there exist a positive integer $N$, a feedback control law which relies only on the current states $x_n(t)\in S^2$ of the first $N$ individual systems, and a time $T$ such that if the initial condition $x(0)$ belongs to some residual set, then the solution $x(t) = (x_n(t))$  of the ensemble system generated by the control law satisfies $\sum_{n = 1}^\infty 2^{-n} \|x_n(t) - v\| \leq \epsilon$ for all $t \geq T$. 
Other relevant works but for {\em finite} ensemble systems include, e.g.,~\cite{ryan2014simultaneous} where the author proposed a feedback control law for stabilizing a finite ensemble of oscillators, and~\cite{guth2023ensemble} where the authors investigated a Riccati based feedback mechanism for stabilizing a finite ensemble of linear systems.

Beyond ensembles of control systems described by ordinary differential equations,   
we mention~\cite{alleaume2023ensembles} where the authors addressed the problem of stabilizing an infinite ensemble of hyperbolic partial differential equations (PDEs). Although the stabilization problems addressed are similar in spirit, the results and the techniques (generalization of PDE backstepping to infinite ensembles) used in~\cite{alleaume2023ensembles} are different from ours.  

Note that the operator $T_k$ is a rank-one perturbation of $A$. Finite rank perturbations of linear operators have been extensively investigated in the literature. However, to the best of the author's knowledge, there does not seem to have solutions to our questions. Perhaps the closest works to ours are~\cite{dobosevych2021spectra,dobosevych2021direct}, where the authors investigated the    
spectra of rank-one perturbations of unbounded self-adjoint operators and the associated pole placement problem.  
Specifically, the authors assumed that there is a constant $d > 0$ such that $a_{n + 1} - a_n \geq d$ for all $n\in \mathcal{I}$, with $\mathcal{I} = \N$ or $\mathcal{I} = \Z$, and both $b$ and $k$ are free to choose. Under this assumptions, they showed that $\{\lambda_n \mid n\in \mathcal{I}\}$ can be the spectrum of $(A + bk)$, for some $b\in X$ and $k\in X^*$, if and only if $\sum_{n \in \mathcal{I}}|\lambda_n - a_n| < \infty$. Their setting is different from ours and, consequently, the necessary and/or sufficient conditions for feasibility of pole placement will be different. For example, we will see in Theorem~\ref{thm:theoremnecessary} that $A$ is necessarily a compact operator. Moreover, in item~1 of Theorem~\ref{thm:necessaryandsufficient}, we will see that the condition $\sum_{n = 1}^\infty a_n < \infty$ is not sufficient for the existence of $b\in X$ and/or $k\in X^*$ such that $\Sigma(A + bk) \subseteq H$. 

Also, as mentioned in Remark~\ref{rmk:strongstabilitysemigroup},  
asymptotic stability of system~\eqref{eq:feedbacksystem} is closely related to strong stability of the semigroups $(\exp(T_k t))_{t\geq 0}$, which has also been extensively investigated. See the recent paper~\cite{chill2020semi}, the survey paper~\cite{chill2007stability}, the book~\cite{engel2000one}, and the references therein. However, many existing results, such as the Arendt-Batty-Lyubich-V\~{u} (ABLV) theorem, assume that the semigroup $(\exp(T_kt))_{t\geq 0}$ is bounded (which implies that the system~\eqref{eq:feedbacksystem} is stable). Upon this hypothesis, the ABLV Theorem and its variations provide sufficient conditions for strong stability of the semi-group. In this paper, we do not make such an assumption, i.e., we do not assume that system~\eqref{eq:feedbacksystem} is stable. But rather, we investigate when this assumption can be satisfied. 

Finally, we mention the problem of {\em simultaneous stabilization}, which has been addressed by Sontag~\cite{sontag1985introduction}, Ghosh~\cite{ghosh1985some}, Blondel, etc.,~\cite{blondel1993sufficient}, Tannenbaum~\cite{tannenbaum2006invariance} and many others. 
We point out that simultaneous stabilization is  different from ensemble feedback stabilization. 
The former deals with the problem of finding a common (or parameter dependent) feedback gain~$k$ such that every individual  linear system $\dot x_\sigma(t) = (A_{\sigma} + b_\sigma k) x_\sigma(t)$, obtained by closing its loop with the feedback control law $u_\sigma(t) = k x_\sigma(t)$, is asymptotically stable, where $\sigma$ is the parameter. 
Note that the control inputs $u_\sigma(t) = kx_\sigma(t)$ for different individual systems are allowed to be {\em different} from each other; in fact, all the closed-loop systems are {\em completely decoupled}. 
This is in contrast with the setting considered in this paper: For ensemble feedback stabilization, we will have to use the {\em same} feedback control law $u(t) = kx(t)$ for all individual systems. 
Given Assumption~\ref{asmp:positivean}, it is necessary that this common control input integrates (in a linear way) the information of all the individual systems, and the resulting feedback ensemble system is coupled through this feedback control law.

\subsection{Notation} 
We gather here key notations used throughout the paper. 

Denote by $\N$ the set of positive integers and $\N_0$ the set of nonnegative integers.  

We use $x = (x_n)$, for $n \in \N$, to denote an infinite sequence. The entries $x_n$'s are complex numbers.  
For $1\leq p < \infty$, let $\ell^p$ be the space of all sequences $x = (x_n)$ such that $\sum_{n = 1}^\infty |x_n|^p < \infty$. 
Let $\ell^\infty$ be the space of bounded sequences, $c\subset \ell^\infty$ be the space of convergent sequences, $c_0\subset c$ be the space of sequences that converge to $0$, and $c_{00}\subset c_0$ be the space of eventually zero sequences, i.e., sequences that have only finitely many nonzero entries. Denote by $\|\cdot \|_{\ell^p}$ the $\ell^p$-norm for all $1\leq p\leq \infty$. We equip $c$, $c_0$, and $c_{00}$ the $\ell^\infty$-norm.

Let $\mathcal{B}(X)$ be the space of all bounded linear operators from $X$ to $X$. Denote by $\|\cdot \|_{\mathcal{B}(X)}$ the operator norm. 
For a given $T\in \mathcal{B}(X)$, let $\Sigma(T)$ be the {\em spectrum} of $T$, $\rP(T):= \C \backslash\Sigma(T)$ be the {\em resolvent set}, and  $R(z,T) := (T - z)^{-1}$, for $z \in \rP(T)$, be the {\em resolvent} of $T$. 
A point $\sigma\in \Sigma(T)$ is {\em discrete} if it is isolated and if the rank of the corresponding Riesz projector $P_\sigma := - \frac{1}{2\pi \mathrm{i}}\oint_{\Gamma} R(z, T)\rd z$  is finite, where $\Gamma$ is a closed rectifiable curve in $\rP(T)$ enclosing only the point $\sigma$.    
Further, let $\Sigma_{\rm disc}(T)$ be the {\em discrete spectrum} of $T$, and $\Sigma_{\rm ess}(T):= \Sigma(T) \backslash \Sigma_{\rm disc}(T)$ be the {\em essential spectrum} of $T$.   

We use $\bfo$ to denote the vector/sequence of all ones or a constant function valued at $1$, and $I$ to denote either the identity matrix or the identity operator. 
Given a finite or an infinite sequence of complex numbers $a_1, a_2, \cdots$, we let $\Diag(a_1, a_2, \cdots)$ be the diagonal matrix or the diagonal operator, with $a_i$ the $ii$th entry.

\section{Main Results}\label{sec:mainresult} 
In this section, we present conditions that are necessary/sufficient for feasibility of placing the poles in the left half plane and for asymptotic  stability of the feedback system.   
We start with the following result: 

\begin{Theorem}\label{thm:theoremnecessary}
Let $(a_n) \in \ell^\infty$, with $a_n > 0$ for $n\in \N$, and $(b_n)\in X$. 
Suppose that there is a $k\in X^*$ such that~\eqref{eq:necessarycondition} is satisfied; then, the following hold: 
\begin{enumerate}
\item $(a_n)\in c_0$ and, moreover, $a_n \neq a_m$ for $n\neq m$; 
\item $b_n\neq 0$ for all $n\in \N$.  
\end{enumerate}
\end{Theorem}

\begin{Remark}\normalfont
    System~\eqref{eq:openloopsystem} is said to be {\em approximately controllable} if for any initial condition $x(0)\in X$, any target $x^*\in X$, any $T > 0$, and any error tolerance $\epsilon > 0$, there exists an integrable function $u:[0,T]\to \C$ such that the solution $x(t)$ generated by~\eqref{eq:openloopsystem} satisfies $\|x(T) - x^*\|_X < \epsilon$. 
    We claim that if $X = c_0$ or $X = \ell^p$, for $1\leq p < \infty$, and if the two items of Theorem~\ref{thm:theoremnecessary} are satisfied, then system~\eqref{eq:openloopsystem} is approximately controllable. 
    To wit, let $k = (k_n)\in X^*$ be such that $kA^m b= 0$, for all $m\in \N_0$. Note that $\omega := (\omega_n:= k_nb_n)\in \ell^1$ and satisfies 
    \begin{equation}\label{eq:omegasolution}
    \sum_{n = 1}^\infty \omega_n a_n^m = 0, \quad \mbox{for all } m\in \N_0.
    \end{equation}
    Since $a\in c_0$ and the $a_n$'s are pairwise distinct, it is known~\cite[Theorem 4]{wermer1952invariant} that $\omega = 0$ is the only solution to~\eqref{eq:omegasolution}. Since $b_n\neq 0$ for all $n\in \N$, we must have that $k = 0$. Thus, the linear span of $A^mb$, for $m\in \N_0$, is a dense subspace of $X$.  By~\cite{triggiani1975controllability},  system~\eqref{eq:openloopsystem} is approximately controllable. \hfill{\qed}
\end{Remark}

For the remainder of the section, we assume that the two items of Theorem~\ref{thm:theoremnecessary}  are satisfied.  
We also assume, without loss of generality, that $(a_n)$ is strictly monotonically decreasing. 
By item~1 of Theorem~\ref{thm:theoremnecessary}, $A$ is a compact operator. Since $T_k$ is a rank-one perturbation of $A$, $T_k$ is compact as well. We  introduce the following definition: 

\begin{Definition}
    A sequence $(\lambda_n)\in c_{0}$ is {\bf feasible} if there is a $k\in X^*$ such that 
    \begin{equation}\label{eq:feasibilityofpoles}
    \Sigma(T_k) = \{\lambda_n \mid n\in \N_0\}, \quad \mbox{where } \lambda_0:= 0.
    \end{equation}
\end{Definition}

Let $c_{H}:= \{(\lambda_n) \in c_{0} \mid \lambda_n \in H \}$.  
We provide below a sufficient condition for $\lambda\in c_H$ to be feasible. 
Note that if $X = c_0$, then $X^* = \ell^1$ and if $X = \ell^p$ for $1\leq p < \infty$, then $X^* = \ell^q$, where $1 < q \leq \infty$ is such that $1/p + 1/q = 1$. In either of these two cases, we can express $k = (k_n) \in X^*$  as 
\begin{equation}\label{eq:generallinearform}
k: x \mapsto kx = \sum_{n = 1}^\infty k_n x_n.
\end{equation}
However, if $X = \ell^\infty$, then $X^*$ contains $\ell^1$ as a proper subspace. In particular, not every $k\in \ell^{\infty *}$ takes the form~\eqref{eq:generallinearform}. For the sufficient condition (Theorem~\ref{thm:main1}) presented below and the consequent results, we will focus only on elements $k\in X^*$ of type~\eqref{eq:generallinearform}.  

To this end, for each $\lambda\in c_H$, we define the sequence $k(\lambda) = (k_n(\lambda))$ as follows: 
\begin{equation}\label{eq:defklambda1}
k_n(\lambda) := -\frac{(a_n - \lambda_n)}{b_n} \prod_{m = 1, m\neq n}^\infty \frac{1 - \lambda_m/a_n}{1 - a_m/a_n}, \quad \mbox{for } n \in \N.
\end{equation} 
The entries $k_n(\lambda)$'s are well defined if the infinite products converge for all $n\in \N$, which can be satisfied if both $(a_n)$ and $(\lambda_n)$ belong to $\ell^1$. However, this condition does not guarantee that $k(\lambda)$ is bounded, not to mention being an element of $X^*$, i.e., being such that 
\begin{equation*}\label{eq:klambdainl1}
k(\lambda): (x_n) \in X \mapsto \sum_{n = 1}^\infty k_n(\lambda) x_n \in \C, 
\end{equation*}
is well defined. 
We will soon provide necessary/sufficient conditions for $k(\lambda)\in X^*$. Before that, we have the following result:

\begin{Theorem}\label{thm:main1} 
Let $\lambda\in c_H$ and $k(\lambda)$ be given as in~\eqref{eq:defklambda1}. 
If $k(\lambda)\in X^*$, then $\lambda$ is feasible and $$\Sigma_{T(k(\lambda))} = \{\lambda_n \mid n\in \N_0\}.$$  
\end{Theorem}


We will now present conditions that are either necessary or sufficient for  $k(\lambda) \in X^*$. A real sequence $(x_n)$, not necessarily bounded, is said to be {\em eventually monotonically decreasing} (resp., {\em eventually monotonically increasing}) if there exists an $N\in \N$ such that $(x_n)$ is monotonically decreasing (resp. monotonically increasing) for $n\geq N$.    

\begin{Theorem}\label{thm:necessaryandsufficient}
The following hold: 
\begin{enumerate}
\item If there exists a $d < 2$ such that $(n^d a_n)_{n\in \N}$ is eventually monotonically increasing, then, regardless of $b\in X$, there does not exist any $\lambda \in c_H$ such that $k(\lambda)\in \ell^\infty$.
\item If there exists a $d > 2$ such that $(n^d a_n)_{n \in \N}$ is eventually monotonically decreasing and if $$\limsup\limits_{n\to\infty}\frac{1}{n}\ln (a_n/|b_n|) \leq 0,$$ 
then for any $\lambda \in c_H$ such that $\lim_{n\to\infty}\lambda_n/a_n = 0$, we have $k(\lambda) \in \ell^1$.
\end{enumerate}
\end{Theorem}

We have so far provided necessary/sufficient conditions for feasibility of placing the poles of the resolvent of $T_k$ in $H$. In fact, if $(a_n)$ and $(b_n)$ satisfy the condition in item~2 of Theorem~\ref{thm:necessaryandsufficient}, then one can choose a feasible $\lambda = (\lambda_n)$ such that $\re(\lambda_n) < 0$ for all $n\in \N$. However,~\eqref{eq:necessarycondition} alone, or, the condition that $\Sigma_{\rm disc}(T_k)$ is contained in the interior of $H$ does not guarantee stability of feedback system~\eqref{eq:feedbacksystem}. Instead, it only implies that if $\|x(t)\|_{X}$ is unbounded as $t\to\infty$, then $\|x(t)\|_{X}$ cannot grow exponentially fast.

We present below a sufficient condition for existence of a $k\in \ell^1$  that renders system~\eqref{eq:feedbacksystem} asymptotically stable. 
Given $(a_n)$ and $(b_n)$, we introduce two new objects: One is the sequence $\pi = (\pi_n)$ defined as follows:  
\begin{equation}\label{eq:defpi}
\pi_n := 2\prod_{m = 1, m\neq n}^\infty \frac{1 + a_m/a_n}{1 - a_m/a_n}, \quad \mbox{for all } n\in \N.
\end{equation}
Note that $k(-a)$ and $\pi$ are related by $k_n(-a) = - a_n\pi_n/b_n$ for all $n\in \N$.   
The other object is an infinite dimensional matrix $\Phi = [\phi_{ij}]_{1 \leq i, j < \infty}$, with the $ij$th entry $\phi_{ij}$ defined as follows: 
\begin{equation}\label{eq:defphi}
\phi_{ij} := \frac{|b_i/b_j|}{ 1 + a_i/a_j }.
\end{equation}
Note that the diagonal entries of $\Phi$ have the same value, given by $\phi_{ii} = 1/2$ for all $i\in \N$. We say that~$\Phi$ {\em spatially exponentially decays} if there exist constants $C>0$ and $\mu\in (0,1)$ such that 
$$
\phi_{ij} \leq C \mu^{|i - j|}, \quad \mbox{for all } i,j\in \N. 
$$
We have the following result: 

\begin{Theorem}\label{thm:sufficientnew} 
Let $\pi$ and $\Phi$ be given as in~\eqref{eq:defpi} and~\eqref{eq:defphi}, respectively. 
Suppose that $\pi \in \ell^\infty$ and that $\Phi$ spatially exponentially decays; then, $k(-a)\in \ell^1$. Moreover, the feedback system
\begin{equation}\label{eq:feedbacksystem2}
    \dot x(t) = T_{k(-a)} x(t),
\end{equation}
with $X$ the state-space, satisfies the following:   
\begin{enumerate}
\item If $X = c_0$ or if $X = \ell^p$ for $1\leq p < \infty$, then system~\eqref{eq:feedbacksystem2} is asymptotically stable.  
\item If $X = \ell^\infty$ of if $X = c$, then system~\eqref{eq:feedbacksystem2} is stable, but not asymptotically stable.
\end{enumerate}
\end{Theorem}

We conclude this section by presenting a sufficient condition for the hypothesis of Theorem~\ref{thm:sufficientnew} to be satisfied. 

\begin{Theorem}\label{thm:sufficient}
Suppose that there exist constants $0 < \nu_0 < \nu_1 < \nu_2 < 1$ such that 
\begin{equation*}
 a_{n+1}/a_n < \nu_0 \quad \mbox{and} \quad  \nu_1 < |b_{n+1}/b_n| < \nu_2, \quad \mbox{for all } n\in \N; 
\end{equation*}
then, $\pi\in \ell^\infty$ and $\Phi$ spatially exponentially decays.  
\end{Theorem}

\section{Proofs of the Main Results}\label{sec:proofs}
In this section, we establish the results presented in Section~\S\ref{sec:mainresult}. There are five subsections, each of which is dedicated to the proof of an individual theorem. 

\subsection{Proof of Theorem~\ref{thm:theoremnecessary}}\label{ssec:theoremnecessary}
The proof relies on the use of the Weinstein–Aronszajn (W-A) formula, which we recall below. 
Let 
$\Delta:= \rP(A) \cup \Sigma_{\rm disc}(A)$.   
Since $T_k$ is a rank-one perturbation of $A$, we have $\Sigma_{\rm ess}(T_k) = \Sigma_{\rm ess}(A)$. 
Thus, $\Delta = \rP(T_k) \cup \Sigma_{\rm disc}(T_k)$. 
For a given $k = (k_n)\in X^*$, we define $\gamma_k:\Delta \to \N_0$ as follows:
\begin{equation*}\label{eq:defmuk}
\gamma_k(z) :=
\begin{cases}
0 & \mbox{if } z\in \rP(T_k), \\
\rk(P_z) & \mbox{if } z\in \Sigma_{\rm disc}(T_k), 
\end{cases}
\end{equation*}
where $\rk(P_z)$ is the rank of the Riesz operator for~$z$.   
We also introduce the following meromorphic function: 
\begin{equation}\label{eq:defh}
h_k(z) := 1 + k (A - z)^{-1} b.
\end{equation} 
The multiplicity function $\nu_h: \Delta \to \Z$ associated with $h_k$ is defined as follows:
\begin{equation*}\label{eq:defnu}
\delta_k(z) := 
\begin{cases}
m & \mbox{if } z \mbox{ is a zero of } h_k \mbox{ with order } m, \\
-m & \mbox{if } z \mbox{ is a pole of } h_k \mbox{ with order } m,  \\
0 & \mbox{otherwise}.
\end{cases}
\end{equation*}
The following result is known~\cite[IV-\S 6]{kato2013perturbation}:

\begin{Lemma}[W-A formula]\label{lem:WAformula}
For any $k \in X^*$, we have 
$$\gamma_{k}(z) = \gamma_0(z) + \delta_k(z), \quad \mbox{for } z\in \Delta.$$
\end{Lemma}

With the lemma above, we prove Theorem~\ref{thm:theoremnecessary}. 

\begin{proof}[Proof of Theorem~\ref{thm:theoremnecessary}] 
Let $k\in X^*$ be such that $\Sigma(T_k) \subseteq H$. We show below that the two items of the theorem must be satisfied. 

\xc{Proof of item~1.} 
Let $\sigma$ be an accumulation point of $(a_n)$. Then, $\sigma\in \Sigma_{\rm ess}(A)$. Since $\Sigma_{\rm ess}(T_k) = \Sigma_{\rm ess}(A)$, we have that $\sigma \in \Sigma_{\rm ess}(T_k)$. By the hypothesis that $\Sigma(T_k)\subseteq H$, it follows that $\sigma$ can only be $0$, which implies that $(a_n)\in c_0$. 
Consequently,  
$\Sigma_{\rm disc}(A) = \{a_n \mid n \in \N\}$  and hence,  $\Delta = \C \backslash \{0\}$.  
We now show that the $a_n$'s are pairwise distinct. 
Consider the meromorphic function $h_k(z)$ given in~\eqref{eq:defh}. The poles of $h_k(z)$ are contained in the set $\{a_n \mid n\in \N\}$. Moreover, if $a_n$, for some $n\in \N$, is a pole of $h_k(z)$, its order is $1$. 
Then, by Lemma~\ref{lem:WAformula} and by the hypothesis, we have that
\begin{equation}\label{eq:inequalityforWAF}
0 = \gamma_k(a_n) = \gamma_0(a_n) + \delta_k(a_n) \geq 0, \quad \mbox{for all }n \in \N,
\end{equation}
where the inequality holds because $\gamma_0(a_n) \geq 1$ and $\delta_k(a_n)\leq 1$. Thus, for~\eqref{eq:inequalityforWAF} to hold, we must have that $\gamma_0(a_n) = \delta_k(a_n) = 1$, which is true if and only if the algebraic multiplicity of $a_n$ is $1$ for all $n\in N$ (i.e., the $a_n$'s are pairwise distinct) and each $a_n$ is a pole of $h_k(z)$.  

\xc{Proof of item~2.} On one hand, we have that $\delta_k(a_n) = 1$ for all $n\in \N$. On the other hand, given that all the $a_n$'s are pairwise distinct, $a_n$ is a pole of $h_k(z)$ only if $b_n \neq 0$.

This completes the proof of Theorem~\ref{thm:theoremnecessary}. 
\end{proof}

In the sequel, we assume that the two items of Theorem~\ref{thm:theoremnecessary} are satisfied. Also, 
 without loss of generality, we assume that  $(a_n)$ is strictly monotonically decreasing.

\subsection{Proof of Theorem~\ref{thm:main1}}\label{ssec:proofofmain1} 
To establish the theorem, we will first  construct a family of eventually zero sequences $k(\lambda; N)$, for $N\in \N$, converging to $k(\lambda)$ as $N\to \infty$, and then use perturbation theory to show that 
$$\Sigma(T_{k(\lambda)}) = \lim_{N\to\infty }\Sigma(T_{k(\lambda;N)}) = \{\lambda_n\mid n\in \N_0\}.$$ 

\xc{Ackermann's formula:} Consider an $N$-dimensional linear control system with single-input: 
\begin{equation}\label{eq:findimsys}
\dot x'(t) = A' x'(t) + b' u(t),  
\end{equation}
where $A'$ and $b'$ are given by
\begin{equation}\label{eq:defA'andb'forfindimsys}
A':= \Diag(a_1,\cdots, a_N) \quad \mbox{and} \quad b':= (b_1,\cdots, b_N).
\end{equation}
Since the $a_n$'s are pairwise distinct and since the $b_n$'s are nonzero, system~\eqref{eq:findimsys} is controllable and hence, pole placement is feasible for~\eqref{eq:findimsys}. 
Given $N$ complex numbers $\lambda_1,\ldots, \lambda_N\in \C$, the Ackermann's formula~\cite{ackermann1972entwurf} provides an explicit expression for a (unique) row vector $k'\in \C^N$ such that $\Sigma(A' + b'k') = \{\lambda_1,\ldots, \lambda_N\}$. The formula is given by 
\begin{equation}\label{eq:ackermann}
k' = -e_N C(A', b')^{-1} q(A'), 
\end{equation}
where $e_N = (0,\ldots,0, 1)\in \R^N$ is a row vector, $C(A', b') := [b', A'b', \cdots, A'^{N-1}b']$ is the controllability matrix associated with $(A',b')$, and $q(z):= \prod_{n = 1}^N(z - \lambda_n)$ is the monic polynomial with $\lambda_n$'s the roots. 
The following result must be known, but we include a proof for completeness of presentation:

\begin{Lemma}\label{lem:explicitexpressionofKN}
The vector $k' = (k'_1,\ldots,k'_N)$ in~\eqref{eq:ackermann} is given by
\begin{equation}\label{eq:defkN}
k'_n = -\frac{(a_n - \lambda_n)}{b_n} \prod_{m = 1, m\neq n}^N \frac{1 - \lambda_m/a_n}{1 - a_m/a_n}, \quad \mbox{for } n = 1,\ldots, N. 
\end{equation}
\end{Lemma}

\begin{proof}
Let $B':= \Diag(b')$ and $V$ be the Vandermonde matrix 
    $$V := [ a_i^{j - 1} ]_{1 \leq i, j\leq N}.$$ 
Then, $C(A',b')$ can be expressed as $C(A',b') = B' V$. 
We compute below $V^{-1}$. Let $p(z)= z^N + \sum_{n = 0}^{N-1} c_{n} z^n$ be the characteristic polynomial of $A'$. Using the coefficients $c_1,\ldots, c_{N-1}$, we define the matrix:  
\begin{equation*}\label{eq:toeplitzmatrix}
L := 
\begin{bmatrix}
c_1 & \cdots & c_{N-1} & 1\\
\vdots & \iddots & \iddots &  \\
c_{N-1} & 1 &  &  \\
1 &  &  & 
\end{bmatrix}, 
\end{equation*}
where the entries below the anti-diagonal are zero. 
Next, let $d_n := \prod_{m = 1, m\neq n}^N (a_n - a_m)$, for $n = 1,\ldots, N$, and $D := \Diag(d_1,\ldots,d_N)$.   
Then, it is known~\cite{chen1981new,man2014inversion} that 
$V^{-1} = L V^\top D^{-1}$.  
By the Ackermann's formula~\eqref{eq:ackermann}, we have that 
$$
k' = - e_N C(A',b')^{-1} q(A') = - e_N LV^\top D^{-1} B'^{-1} q(A') = -\bfo D^{-1}B'^{-1} q(A'), 
$$
where $\bfo\in \R^N$ is a row vector of all ones and the fact that $e_N LV^\top = \bfo$ follows directly from computation. 
Expression~\eqref{eq:defkN} is then an immediate consequence of the above equation. 
\end{proof}

Note that the expression of $k(\lambda)$ given in~\eqref{eq:defklambda1} is a straightforward extension of~\eqref{eq:defkN} by letting $N\to\infty$. Thus, Theorem~\ref{thm:main1} can be viewed as an extension of the Ackermann's formula for the case where $A$ is an (infinite dimensional) diagonal operator.

\xc{Convergent sequence of feedback gains.} 
Recall that $c_{00}$ is the space of eventually zero sequences. 
For every $N\in \N$, we define an element $k(\lambda; N)\in c_{00}$ as follows: 
\begin{equation}\label{def:approximationofklambda}
k_n(\lambda;N) := 
\begin{cases}
k'_n & \mbox{if } 1\leq n \leq N,  \\
0 & \mbox{otherwise},  
\end{cases}
\end{equation}
where $k'_n$ is given as in~\eqref{eq:defkN}. 
Note that $k(\lambda;N)$ is {\it not} obtained by truncating $k(\lambda)$, i.e., $k_n(\lambda, N) \neq k_n(\lambda)$ for $n = 1,\ldots, N$. We have the following result: 

\begin{Proposition}\label{prop:convergenceofkN}
Suppose that  $k(\lambda)$, for $\lambda\in c_H$, belongs to $\ell^q$ for some $1 \leq q < \infty$; then, 
$$\lim_{N\to\infty}\|k(\lambda;N) - k(\lambda)\|_{\ell^q} = 0.$$  
\end{Proposition}

\begin{proof}
For ease of notation, we use $k$ and $k(N)$ to denote $k(\lambda)$ and $k(\lambda; N)$, respectively. 
For any $\epsilon > 0$, we exhibit an $N_\epsilon$ such that 
\begin{equation}\label{eq:defNstar}
\|k(N) - k\|_{\ell^q} < \epsilon, \quad \mbox{for all } N \geq N_\epsilon.
\end{equation}
Since $k(\lambda)\in \ell^q$, there is an $N'\in \N$ such that 
\begin{equation}\label{eq:defN'}
\sum_{n = N'+1}^\infty |k_n|^q < \epsilon^q/2^{q+1}.
\end{equation} 
Next, for $n\leq N$, we define
$$r_n(N) :=  \prod_{m = N + 1}^\infty \frac{1 - a_m/a_n}{ 1 - \lambda_m/a_n}.$$
By~\eqref{eq:defkN} and~\eqref{eq:defklambda1}, we have $k_n(N) = r_n(N) k_n$. 
Since $(a_n)$ is strictly monotonically decreasing and since all the $a_n$'s are positive, we have $0 < 1 - a_m/a_n < 1$ for $m > n$.   
Also, since $\re(\lambda_n)\leq 0$ for $n\in \N$, we have $|1 - \lambda_m/a_n| \geq 1$ for  $m, n\in \N$. 
It follows that 
\begin{equation}\label{eq:rnlessthan1}
|r_n(N)| < 1, \quad \mbox{for } n \leq N.  
\end{equation} 
Further, by the hypothesis that $\lim_{N\to\infty} k_n(N) = k_n$, we have 
$\lim_{N\to\infty} r_n(N) = 1$ for all $n\in \N$. 
Thus, given the integer $N'$ that satisfies~\eqref{eq:defN'}, there exists an $N'' \geq N'$ such that   
\begin{equation}\label{eq:boundrN}
| 1- r_{n}(N)|^q \leq \frac{\epsilon^q}{2\|k\|_{\ell^q}^q}, \quad \mbox{for all } n \leq N' \mbox{ and for all } N \geq N''.
\end{equation}
We claim that $N_\epsilon:= N''$ satisfies~\eqref{eq:defNstar}. To wit, for any $N \geq N_\epsilon$, we have
\begin{align*}
\|k - k(N)\|_{\ell^q}^q  & = \sum_{n = 1}^{N} |k_n - k_n(N)|^q + \sum_{n = N + 1}^{\infty} |k_n|^q \\
& = \sum_{n = 1}^{N'} |k_n |^q |1 - r_n(N)|^q + \sum_{n = N'+1}^{N} |k_n |^q |1 - r_n(N)|^q + \sum_{n = N + 1}^{\infty} |k_n|^q \\
& \leq \frac{\epsilon^q}{2 \|k\|_{\ell^q}^q} \sum_{n = 1}^{N'} |k_n|^q +  2^q \sum_{n = N'+1}^\infty |k_n|^q  \\
& \leq \frac{\epsilon^q}{2 \|k\|_{\ell^q}^q} \sum_{n = 1}^{\infty} |k_n|^q + 2^q \frac{\epsilon^q}{2^{q+1}} \leq \epsilon^q/2 + \epsilon^q /2 = \epsilon^q,
\end{align*}
where we have used~\eqref{eq:rnlessthan1} and~\eqref{eq:boundrN} to establish the first inequality and used~\eqref{eq:defN'} to establish the second inequality. This completes the proof. 
\end{proof}

With Proposition~\ref{prop:convergenceofkN}, we prove Theorem~\ref{thm:main1}. 

\begin{proof}[Proof of Theorem~\ref{thm:main1}] 
For convenience, we still use $k$ and $k(N)$ to denote $k(\lambda)$ and $k(\lambda; N)$, respectively. 
By its construction~\eqref{eq:defklambda1} and by the hypothesis of the theorem, we have $k \in \ell^q$ for some $1\leq q\leq \infty$.     More specifically, we have the following two cases: 
(1) If $X = c$, or, $X = c_0$, or, $X = \ell^p$, for $1 < p \leq \infty$,  then $k \in \ell^q$ for some $1\leq q < \infty$; (2) If $X = \ell^1$, then $k \in \ell^\infty$. We establish the theorem for these two cases:

\xc{Case 1: $k \in \ell^q$ for $1\leq q < \infty$.} 
First, note that $$\|T_{k(N)} - T_{k}\|_{\mathcal{B}(X)} \leq \|b\|_{X} \|k(N) - k\|_{\ell^q}.$$ 
It then follows from Proposition~\ref{prop:convergenceofkN} that 
$\lim_{N\to \infty}\|T_{k(N)} - T_k\|_{\mathcal{B}(X)} = 0$.
For each $N$, let $A'(N)\in \R^{N\times N}$ and $b'(N)\in \R^N$ be given as in~\eqref{eq:defA'andb'forfindimsys}, $k'(N)\in \R^N$ be given as in~\eqref{eq:defkN}, and $$T'(N):= A'(N) + b'(N)k'(N).$$ 
By the definition of $k(N)$ as in~\eqref{def:approximationofklambda}, each operator $T_{k(N)}$ is in the block lower triangular form, with  $T'(N)$ and $A''(N):= \Diag(a_{n+1},a_{n+2}, \cdots)$ the diagonal blocks. By Lemma~\ref{lem:explicitexpressionofKN}, the eigenvalues of $T'(N)$ are $\lambda_1,\ldots,\lambda_N$. It follows that  
\begin{equation}\label{eq:spectrumofapproximation}
\Sigma(T_{k(N)}) = \{\lambda_n \mid 0\leq n \leq N \} \cup \{a_n \mid n > N \}.
\end{equation}

We now use~\eqref{eq:spectrumofapproximation} and perturbation theory to show that $\Sigma(T_k) = \{\lambda_n \mid n\in \N_0\}$.   
By continuity of a finite system of eigenvalues~\cite[IV-\S 3]{kato2013perturbation}, if $\lambda\in \Sigma_{\rm disc}(T_{k})$, then for sufficiently large $N$, $T_{k(N)}$ has an isolated eigenvalue arbitrarily close to $\lambda$. Since $\lim_{n\to\infty} a_n = 0$,  by~\eqref{eq:spectrumofapproximation} it must hold that $\lambda = \lambda_n$ for some $n$. The above arguments imply that $\Sigma(T_k) \subseteq \{\lambda_n \mid n\in \N_0\}$. Conversely, we claim that every $\lambda_n$ belongs to $\Sigma_{\rm disc}(T_k)$.  
Suppose not, say $\lambda_n \in \rP(T_k)$; then, by the fact that 
the resolvent $R(z,T)$, for $z\in \C$ and $T\in \mathcal{B}(X)$, is analytic~\cite[IV-\S 3]{kato2013perturbation} at $(\lambda_n,T_k)$, we must have that $\lambda_n\in \rP(T_{k(N)})$ for sufficiently large $N$, which contradicts~\eqref{eq:spectrumofapproximation}. This completes the proof for Case 1.

\xc{Case 2: $k\in\ell^\infty$.} In this case, $k(N)$ may not converge, in the $\ell^\infty$-sense, to $k$ and hence, some additional efforts are needed to establish the result. To this end, let $\tilde k$ and $\tilde k(N)$, for $N\in \N$, be defined as follows: 
$$\tilde k  := (k_n b_n) \quad \mbox{and} \quad \tilde k(N) := (k_n(N)b_n).$$
Since $k\in \ell^\infty$ and $b\in \ell^1$, we have that $\tilde k \in \ell^1$. By Proposition~\ref{prop:convergenceofkN}, $\tilde k(N)$ converges, in the $\ell^1$-sense, to $\tilde k$. 
Next, define the operators $\tilde T: \ell^\infty \to \ell^\infty$ and $\tilde T(N): \ell^\infty \to \ell^\infty$, for $N\in \N$, as
$$
\tilde T := A + \bfo \tilde k \quad \mbox{and} \quad \tilde T(N) := A + \bfo \tilde k(N).
$$
Then, $\tilde T(N)$ converges to $\tilde T$. Note that~\eqref{eq:spectrumofapproximation} still holds for $\Sigma(\tilde T(N))$. Thus, by the same arguments as used in Case 1, we have $\Sigma(\tilde T) = \{\lambda_n | n\in \N_0\}$. 
It now remains to show that $\Sigma(T_k) = \Sigma(\tilde T)$. 

We first show that $\Sigma(T_k) \subseteq \Sigma(\tilde T)$. 
Let $\lambda\neq 0$ be an eigenvalue of $T_k$. We claim that $\lambda \neq a_n$ for any $n\in \N$. To wit, first note that the meromorphic function $h_k(z)$, for $z\in \C\backslash\{0\}$, introduced in~\eqref{eq:defh} is given by
$$
h_k(z) = 1 + \sum_{n = 1}^\infty \frac{k_n b_n}{a_n - z}.
$$
By the definition~\eqref{eq:defklambda1} of $k$, $k_n\neq 0$ for $n\in \N$. Since $b_n \neq 0$, we have that $k_n b_n \neq 0$ for all $n\in \N$. Thus, the poles of $h_k(z)$ are $a_n$, for $n\in \N$, the orders of which are~$1$. It follows from the W-A formula that $$\gamma_k(a_n) = \gamma_0(a_n) + \delta_k(a_n) = 1 - 1 = 0,$$ 
which establishes the claim. 
Now, let $v$ be an eigenvector of $T_k$ corresponding to $\lambda$. 
Entry-wise we have
\begin{equation}\label{eq:eigenvectorfork}
    a_n v_n + (kv) b_n = \lambda v_n, \quad \mbox{for } n\in \N. 
\end{equation}
Since $\lambda \neq a_n$ for any $n\in \N$, $kv \neq 0$. One can thus re-scale $v$ such that $kv = 1$ and, correspondingly, 
$v := (b_n/(a_n - \lambda)) \in \ell^1$. 
Now, let $\tilde v:= (1/(a_n - \lambda_n))\in \ell^\infty$. It should be clear that $\tilde k \tilde v = k v$. Thus, by~\eqref{eq:eigenvectorfork}, we obtain that
\begin{equation*}\label{eq:eigenvectorfortildek}
a_n \tilde v_n + \tilde k \tilde v = \lambda \tilde v_n, \quad \mbox{for } n\in \N, 
\end{equation*}
which implies that $\lambda$ is an eigenvalue of $\tilde T$ with $\tilde v$ a corresponding eigenvector.  

Conversely, let $\lambda \neq 0$ be an eigenvalue of $\tilde T$. We show that $\lambda \in \Sigma(T_k)$. Since $\Sigma(\tilde T) = \{\lambda_n \mid n\in\N_0 \}$,  $\lambda \neq a_n$ for any $n\in \N$. Thus, one can use the same arguments to conclude that $\tilde v = (1/(a_n - \lambda))$ is an eigenvector of $\tilde T$ corresponding to $\lambda$ and, consequently,  $v = (b_n/(a_n - \lambda))$ is an eigenvector of $T_k$ corresponding to~$\lambda$. This completes the proof for Case~2.         
\end{proof}

\subsection{Proof of Theorem~\ref{thm:necessaryandsufficient}}
Given $(a_n)$ and $(\lambda_n)$, we define two sequences $(\alpha_n)$ and $(\beta_n)$ as follows:   
\begin{equation*}\label{eq:defalphabeta}
\alpha_n := \frac{1}{n} \sum_{m = 1,m\neq n}^\infty \ln \left | 1 - a_m/a_n \right |  
\quad \mbox{and} \quad 
\beta_n := \frac{1}{n}  \sum_{m = 1}^\infty \ln \left | 1 - \lambda_m/a_n \right |.
\end{equation*}
It should be clear that 
\begin{equation*}\label{eq:relaabk}
\ln |k_n(\lambda)| = n(\beta_n - \alpha_n) + \ln|a_n/b_n|.
\end{equation*}

We establish below relevant properties about the asymptotic behaviors of $(\alpha_n)$ and of $(\beta_n)$, which  will be essential for the proof of Theorem~\ref{thm:necessaryandsufficient}. 
We start by introducing the function $\zeta: (1,\infty) \to \R$ defined as follows: 
\begin{equation}\label{eq:deffunctionofphi}
\zeta(d):= \int_0^1 \ln (1/z^d - 1) \rd z + \int_1^\infty \ln (1 - 1/z^d) \rd z. 
\end{equation}
We need the following result:

\begin{Lemma}\label{lem:expressionoftheta}
The function $\zeta(d)$ is strictly monotonically increasing and satisfies the following properties: 
\begin{equation}\label{eq:propertiesoftheta}
        \lim_{d\to 1^+} \zeta(d) = -\infty, \quad \zeta(2) = 0, \quad \mbox{and} \quad \lim_{d\to\infty} \zeta(d)/d = 1,        
    \end{equation}
\end{Lemma}

\begin{proof}
It should be clear that $\zeta(d)$ is strictly monotonically increasing. We establish below~\eqref{eq:propertiesoftheta}. 
We compute the two terms in~\eqref{eq:deffunctionofphi}. 
For the first term, we have that
\begin{align}
\int_0^1 \ln (1/z^d - 1) \rd z & = - d \int_0^1 \ln z \rd z + \int_0^1 \ln (1 - z^d) \rd z  \nonumber \\
& = d - \sum_{m = 1}^\infty \int_0^1 \frac{z^{md}}{m} \rd z  
= d - \sum_{m = 1}^\infty \frac{1}{m(md+1)}. \label{eq:firstterm}
\end{align}
For the second term, we have that
\begin{equation}\label{eq:secondterm}
\int_1^\infty \ln (1 - 1/z^d) \rd z  = - \sum_{m = 1}^\infty \int_1^\infty \frac{1}{m z^{md}} \rd z = - \sum_{m = 1}^\infty \frac{1}{m(md-1)}.
\end{equation}
Summing the two terms~\eqref{eq:firstterm} and~\eqref{eq:secondterm}, we obtain that
$$
\zeta(d) = d - \sum_{m = 1}^\infty \left [ \frac{1}{m(md+1)} + \frac{1}{m(md-1)} \right ] = d - 2d \sum_{m = 1}^\infty  \frac{1}{m^2d^2 - 1}.
$$
It should be clear that $\lim_{d\to 1^+} \zeta(d) = -\infty$ and that $\lim_{d\to \infty} \zeta(d)/d = 1$. 
It now remains to show that $\zeta(2) = 0$. We have that  
$$
\zeta(2) = 2 - 4\sum_{m = 1}^\infty \frac{1}{4m^2 - 1} = 2 - 2\sum_{m = 1}^\infty \left [ \frac{1}{2m - 1} - \frac{1}{2m + 1} \right ] = 0,
$$
which completes the proof.
\end{proof}

With Lemma~\ref{lem:expressionoftheta}, we state the following result for $(\alpha_n)$:

\begin{Proposition}\label{prop:condforalpha}
    Let $d > 1$ and $\zeta(d)$ be given as in~\eqref{eq:deffunctionofphi}. 
    Then, the following hold: 
    \begin{enumerate}
    \item If $(n^d a_n)$ is eventually monotonically increasing, then 
    $\limsup\limits_{n\to\infty} \alpha_n \leq \zeta(d)$. 
    \item If $(n^d a_n)$ is eventually monotonically decreasing, then 
    $\liminf\limits_{n\to\infty} \alpha_n \geq \zeta(d)$. 
    \end{enumerate}
\end{Proposition}
  
Next, for $(\beta_n)$, we have the following result: 

\begin{Proposition}\label{prop:condforbeta}
    If there exists a $d>1$ such that $a_n = 1/n^d$ for all $n\in \N$ and if $\lim_{n\to\infty} \lambda_n/a_n = 0$, then $\lim_{n\to\infty} \beta_n = 0$.
\end{Proposition}

Theorem~\ref{thm:necessaryandsufficient} is now a consequence of Propositions~\ref{prop:condforalpha} and~\ref{prop:condforbeta}:

\begin{proof}[Proof of Theorem~\ref{thm:necessaryandsufficient}] 
We establish below the two items of the theorem. 

\xc{Proof of item 1.} 
Since $\lambda \in c_H$, we have $|1 - \lambda_m/a_n| \geq 1$ for any $n, m\in \N$ and hence, 
$|k_n(\lambda)| \geq |k_n(0)|$ for all $n \in \N$. 
Thus, if $k(\lambda)$ is an element of  $\ell^\infty$, then so is $k(0)$. Conversely, if $k(0)$ is not an element of $\ell^\infty$, then there does not exist any $\lambda\in c_H$ such that $k(\lambda)\in \ell^\infty$. 

Let $1 < d < 2$ be such that $(n^d a_n)$ is eventually monotonically increasing. This, in particular, implies that $1/a_n  = O(n^d)$. Combining this fact with item~1 of Proposition~\ref{prop:condforalpha}, we obtain that
$$
| 1/k_n(0) | = |b_n e^{n\alpha_n}/a_n| \leq  \|b\|_{X} \left | e^{ n \alpha_n} /a_n\right |   = O\left (n^d e^{\zeta(d)n} \right ). 
$$
Since $d < 2$, by Lemma~\ref{lem:expressionoftheta}, we have $\zeta(d) < 0$ and hence, 
$\lim_{n\to\infty} n^d e^{\zeta(d)n} = 0$.   
It follows that $\lim_{n\to\infty}  |k_n(0)| = \infty$, so $k(0)\not\in \ell^\infty$.

\xc{Proof of item~2.} For convenience, let $\rho_n := |\lambda_n|/a_n$. 
We introduce the function
$$
f_n(x) := \frac{1 + \rho_n x}{|x - 1|}, \quad \mbox{for } x\in (0,1)\cup (1,\infty).  
$$
Then, we have
\begin{equation*}\label{eq:zetaratio}
\frac{| 1 - \lambda_m /a_n |}{|1 - a_m/a_n|} \leq \frac{1 + \rho_m a_m/a_n}{|1 - a_m/a_n|} = f_m(a_m/a_n). 
\end{equation*}
Let $d > 2$ and $n_0\in \N$ be such that $(n^d a_n)$ is monotonically decreasing for $n\geq n_0$. Then,  
$$
a_m/a_n \geq n^d/m^d > 1, \quad \mbox{for } n > m \geq n_0, 
$$
and
$$
a_m/a_n \leq n^d/m^d < 1, \quad \mbox{for } m > n \geq n_0. 
$$
Note that $f_n(x)$ is monotonically increasing over $(0,1)$ and is monotonically decreasing over $(1, \infty)$. It follows that for any $n, m \geq n_0$, with $n\neq m$, we have 
$$
f_m(a_m/a_n) \leq f_m(n^d/m^d). 
$$
The above arguments then imply that 
\begin{equation}\label{eq:upperboundforknlambda}
|k_n(\lambda)| = \frac{|a_n - \lambda_n|}{|b_n|} \prod_{m = 1,m\neq n}^\infty \frac{|1 - \lambda_m/a_n|}{|1 - a_m/a_n|} \leq \frac{(1 + \rho_n)a_n}{|b_n|} \prod_{m = 1,m\neq n}^\infty \frac{1 + \rho_m n^d/m^d}{|1 - n^d/m^d|}. 
\end{equation}
We now show that the rightmost expression of~\eqref{eq:upperboundforknlambda} decays exponentially fast as $n\to\infty$.  
First, by Lemma~\ref{lem:expressionoftheta} and by item~2 of Proposition~\ref{prop:condforalpha}, we have  
$$
\liminf\limits_{n\to\infty}\frac{1}{n}\sum_{m = 1,m\neq n}^\infty \ln |1 - n^d/m^d| \geq \zeta(d) > 0. 
$$
Next, by Proposition~\ref{prop:condforbeta}, we have
$$
\lim_{n\to\infty}\frac{1}{n} \sum_{m = 1,m\neq n}^\infty \ln (1 + \rho_m n^d/m^d ) = 0.
$$
Further, by the hypothesis of the theorem, $\limsup\limits_{n\to\infty} \frac{1}{n}\ln (a_n/|b_n|) \leq 0$ and $\lim_{n\to\infty} \rho_n = 0$.   
The above arguments combined imply that 
$\limsup\limits_{n\to\infty} \frac{1}{n}\ln |k_n(\lambda)| < 0$, and hence  $k(\lambda)\in \ell^1$. 
\end{proof}

For the remainder of the subsection, we establish Propositions~\ref{prop:condforalpha} and~\ref{prop:condforbeta}.  

\subsubsection{Proof of Proposition~\ref{prop:condforalpha}}
We establish below the two items of the proposition. 

\xc{Proof of item 1.} 
Let $n_0\in \N$ be such that $(n^d a_n)$ is monotonically increasing for $n\geq n_0$. For each $n \geq n_0$, we define
\begin{equation}\label{eq:deftildealpha}
\tilde \alpha_n := \alpha_n - \frac{1}{n}\sum_{m = 1}^{n_0} \ln (a_m/a_n - 1).
\end{equation}
We claim that 
$\lim_{n\to\infty}(\alpha_n - \tilde \alpha_n) = 0$.   
Since $(n^d a_n)$ is eventually monotonically increasing, there exists a $\delta > 0$ such that $n^d a_n \geq \delta$ for any $n\in \N$ and hence, 
\begin{equation}\label{eq:boundforlogan}
-\ln a_n \leq  - \ln \delta + d \ln n, \quad \mbox{for } n\in \N.
\end{equation}
It follows that for each $m = 1,\ldots, n_0$, 
$$
0\leq \lim_{n\to\infty}\frac{1}{n} \ln(a_m/a_n - 1) \leq \lim_{n\to\infty}\frac{1}{n}(\ln a_m - \ln a_n) \leq \lim_{n\to\infty} \frac{1}{n} (\ln a_m - \ln \delta + d \ln n)= 0,
$$
where the third inequality follows from~\eqref{eq:boundforlogan}. Thus
\begin{equation}\label{eq:differencebetweengammas}
\lim_{n \to\infty} (\alpha_n - \tilde \alpha_n) =  \frac{1}{n}\sum_{m = 1}^{n_0} \ln(a_m/a_n - 1) = 0. 
\end{equation}
Now, for $\tilde \alpha_n$, we have that 
\begin{align}
\tilde \alpha_n & = \frac{1}{n} \sum_{m = n_0}^{n-1} \ln ( a_m/a_n 
 - 1 ) + \frac{1}{n} \sum_{m = n + 1}^\infty \ln ( 1 - a_m/a_n ) \vspace{.2cm} \nonumber\\ 
 & = \frac{1}{n} \sum_{m = n_0}^{n-1} \ln \left [ \frac{n^d}{m^d} \frac{m^d a_m}{n^d a_n} 
 - 1\right ] + \frac{1}{n} \sum_{m = n + 1}^\infty \ln \left [ 1 - \frac{n^d}{m^d} \frac{m^d a_m}{n^d a_n}  \right ] \vspace{.2cm}\nonumber \\
& \leq \frac{1}{n} \sum_{m = n_0}^{n-1} \ln ( n^d/m^d  
 - 1 ) + \frac{1}{n} \sum_{m = n + 1}^\infty \ln ( 1 - n^d/m^d), \label{eq:boundfortildealpha}
\end{align}
where the inequality follows from the fact that $(n^d a_n)$ monotonically increases for $n \geq n_0$ and the fact that $\ln(x)$ is a monotonically increasing function. 
Note that for $m \leq n -1$, 
$$
\ln (n^d/m^d - 1 ) \leq \int_{m-1}^{m} \ln (n^d/x^d - 1) \rd x = n \int_{(m-1)/n}^{m/n} \ln (1/z^d - 1) \rd z, 
$$
and for $m \geq n+1$, 
$$
\ln (1 - n^d/m^d ) \leq \int_{m}^{m+1} \ln (1 - n^d/x^d) \rd x = n \int^{(m+1)/n}_{m/n} \ln (1 - 1/z^d ) \rd z.  
$$
Using the above two inequalities, we proceed with~\eqref{eq:boundfortildealpha} and have that
$$
\tilde \alpha_n \leq \int_{(n_0-1)/n}^{(n - 1)/n} \ln (1/z^d - 1) \rd z + \int_{(n + 1)/n}^\infty \ln (1 - 1/z^d) \rd z < \zeta(d), \quad \mbox{for all } n\geq n_0. 
$$ 
Item~1 of the proposition then follows from~\eqref{eq:differencebetweengammas}. 

\xc{Proof of item 2.} 
Let $n_0\in \N$ be such that $(n^d a_n)$ is monotonically decreasing for $n \geq n_0$. Further, let $n_1\in \N$ be such that 
$a_{n_0}/a_{n_1} \geq 2$.   
For each $n \geq n_1$, we still let $\tilde \alpha_n$ be given as in~\eqref{eq:deftildealpha}.  
By the fact that $(a_n)$ is monotonically decreasing and the choice of $n_1$,  we have that $\ln (a_m/a_n - 1) \geq 0$ for $m = 1,\ldots, n_0$ and for $n \geq n_1$. Thus,  
\begin{equation}\label{eq:comparegammaandtildegamma}
\alpha_n \geq \tilde \alpha_n, \quad \mbox{for } n \geq n_1.
\end{equation} 
Similar to the arguments used in the proof of item~1, we have that 
\begin{align}
\tilde \alpha_n & = \frac{1}{n} \sum_{m = n_0}^{n-1} \ln ( a_m/a_n 
 - 1 ) + \frac{1}{n} \sum_{m = n + 1}^\infty \ln (1 - a_m/a_n ) \vspace{.2cm} \nonumber\\ 
& \geq \frac{1}{n} \sum_{m = n_0}^{n-1} \ln ( n^d/m^d  
 - 1 ) + \frac{1}{n} \sum_{m = n + 1}^\infty \ln (1 - n^d/m^d ) \vspace{.2cm} \nonumber \\
& \geq \int_{n_0/n}^1 \ln (1/z^d - 1) \rd z + \int_{1}^\infty \ln (1 - 1/z^d)\rd z. \nonumber
\end{align}
Taking the limit, we obtain that
$\liminf_{n\to\infty } \tilde \alpha_n \geq \zeta(d)$. Item~2 then follows from~\eqref{eq:comparegammaandtildegamma}.   
\hfill{\qed}

\subsubsection{Proof of Proposition~\ref{prop:condforbeta}}  
First, note that for $\lambda \in c_H$, we have $|1 - \lambda_m/a_n| \geq 1$ for all $m, n \in\N$. Thus, $\beta_n \geq 0$ for all $n\in \N$. 
Next, we introduce the function $\xi: (1,\infty)\to \R$ defined as follows: 
$$
\xi(d):= \int_0^\infty \ln (1 + 1/z^d) \rd z.
$$
Using the same technique as in the proof of Lemma~\ref{lem:expressionoftheta}, we can express $\xi(d)$ as follows:  
$$
\xi(d) = d + 2d \sum_{m = 1}^\infty  \frac{(-1)^{m-1}}{m^2 d^2 - 1}.
$$

Let $\rho_n:= |\lambda_n|/a_n$. Since $\lim_{n\to\infty}\rho_n = 0$, there exists a $\bar \rho > 0$ such that 
$\rho_n \leq \bar \rho$, for $n\in \N$.   
For any $\epsilon > 0$, let $N_\epsilon$ be such that $\rho_n \leq \epsilon$ for all $n > N_\epsilon$. Then,  
\begin{align}
\beta_n & \leq \frac{1}{n}\sum_{m = 1}^{\infty} \ln (1 +  |\lambda_m|/a_n) \nonumber \\
& \leq \frac{1}{n} \sum_{m = 1}^{N_\epsilon} \ln (1 + \bar \rho a_m/a_n) + \frac{1}{n} \sum_{m = N_\epsilon + 1}^\infty \ln (1 + \epsilon a_m/a_n) \nonumber \\
& \leq \frac{N_\epsilon}{n} \ln (1 + \bar \rho n^d) + \frac{1}{n} \sum_{m = 1}^\infty \ln (1 + \epsilon n^d/m^d). \label{eq:twotermsbeta}  
\end{align}

Now, we evaluate the two terms in~\eqref{eq:twotermsbeta} in the asymptotic regime.   
For the first term, we have
$$
\lim_{n\to\infty} \frac{N_\epsilon}{n} \ln (1 + \bar \rho n^d)  = 0.
$$
For the second term, 
\begin{align}
\frac{1}{n} \sum_{m = 1}^\infty \ln (1 + \epsilon n^d/m^d) & \leq  \frac{1}{n}\int_0^\infty \ln (1 + \epsilon n^d/x^d ) \rd x \notag \\
& = \epsilon^{\frac{1}{d}} \int_0^\infty \ln (1 + 1/z^d) \rd z  = \epsilon^{\frac{1}{d}} \xi(d). \notag
\end{align}
The above arguments imply that 
$$
0\leq \limsup\limits_{n\to\infty }\beta_n \leq \epsilon^{\frac{1}{d}} \xi(d), \quad \mbox{for any } \epsilon > 0. 
$$
We thus conclude that $\lim_{n\to\infty} \beta_n = 0$. \hfill{\qed}

\subsection{Proof of Theorem~\ref{thm:sufficientnew}}
The proof takes a few steps.  
We will first show that $k(-a)\in \ell^1$ and, hence, $T_{k(-a)}\in \mathcal{B}(X)$. This part is relatively easy.  
The major efforts of the proof will be in establishing the stability properties of the feedback system: 
\begin{equation}\label{eq:reproducek-a}
\dot x(t) = T_{k(-a)} x(t).
\end{equation} 
Specifically, to establish the two items of Theorem~\ref{thm:sufficientnew},  we will introduce a new linear dynamical system, termed the $y$-system, obtained by a linear transformation of~\eqref{eq:reproducek-a}. The $y$-system and the original system~\eqref{eq:reproducek-a} share the same stability properties. Furthermore, we will show that the infinitesimal generator associated with the $y$-system is diagonalizable via a bounded, invertible operator. We will then use this fact to establish the desired stability properties of the $y$-system (more precisely, the diagonalized version of the $y$-system).

For ease of presentation, but without loss of generality, we assume for the remainder of the subsection that $b_n > 0$ for all $n\in \N$.

\subsubsection{Proof that $k(-a)\in \ell^1$}\label{ssec:feedbackgain} 
Recall that $\pi$ is defined in~\eqref{eq:defpi}.  
By its definition, we have that $k_n(-a) = - a_n \pi_n /b_n$ for all $n\in \N$. 
By the hypothesis of Theorem~\ref{thm:sufficientnew}, $\pi\in \ell^\infty$. Thus, to prove that $k(-a)\in \ell^1$, it suffices to show that $(a_n/b_n)\in \ell^1$. 
Let $\Phi = [\phi_{ij}]_{1\leq i, j <\infty}$ be given as in~\eqref{eq:defphi}. With $b_n > 0$ for all $n\in \N$, we have that 
$$
\phi_{ij} = \frac{b_i/b_j}{1 + a_i/a_j}. 
$$
By the hypothesis of the theorem, $\Phi_{ij}$ spatially exponentially decays. Let $C > 0$ and $\mu \in (0,1)$ be such that $\phi_{ij} \leq C \mu^{|i - j|}$. 
Then, for any $i\in \N$, we have 
\begin{equation}\label{eq:amdividebml1first}
\sum_{j = 1}^\infty \phi_{ij} = \sum_{j=1}^i \phi_{ij} + \sum_{j = i+1}^\infty \phi_{ij} \leq C \sum_{j=1}^i \mu^{i - j} + C\sum_{j = i+1}^\infty \mu^{j - i} \leq C\frac{1 + \mu}{1 - \mu} =: \kappa.
\end{equation}
Also, note that 
\begin{equation}\label{eq:amdividebml1}
\sum_{j = 1}^\infty \phi_{ij} = \sum_{j = 1}^\infty \frac{b_i/b_j}{1 + a_i/a_j} = \sum_{m = 1}^\infty \frac{a_j}{b_j} \frac{b_i}{a_j + a_i} \geq \frac{b_i}{a_1 + a_i} \sum_{j = 1}^\infty \frac{a_j}{b_j}. 
\end{equation}
Thus, by~\eqref{eq:amdividebml1first} and~\eqref{eq:amdividebml1}, we have
$$
\|(a_n/b_n)\|_{\ell^1} \leq \kappa \frac{a_1 + a_i}{b_i}, \quad \mbox{for all } i\in \N.  
$$
We conclude from above that $(a_n/b_n)\in \ell^1$.  \hfill{\qed}

\begin{Remark}\normalfont
    Similar to~\eqref{eq:amdividebml1first}, we have that for any $j\in \N$, 
    \begin{equation}\label{eq:amdividebml1column}
        \sum_{i = 1}^\infty \phi_{ij} = \sum_{i=1}^j \phi_{ij} + \sum_{i = j+1}^\infty \phi_{ij} \leq C \sum_{i=1}^j \mu^{j-i} + C\sum_{i = j+1}^\infty \mu^{i - j} \leq \kappa.
    \end{equation}
    Thus, the rows $(\phi_{ij})_{j\in \N}$ and the columns $(\phi_{ij})_{i\in \N}$ of $\Phi$ are elements of $\ell^1$. \hfill{\qed}
\end{Remark}

\subsubsection{Translation to the $y$-system}\label{ssec:newsystem}
Let $Y$ be the Banach space defined as follows:
\begin{equation}\label{eq:defXbXb}
Y:= \{ (y_n) \mid  (b_n y_n) \in X \}.  
\end{equation}
For any $y = (y_n)\in Y$, we set $$\|y\|_{Y} := \|(b_n y_n)\|_{X}.$$

Let $B: Y \to X$ be the diagonal operator defined by   
$B: (y_n) \mapsto (b_n y_n)$.    
Then, the operator $B^{-1}: X \to Y$  that sends  
$(x_n)$ to $(x_n/b_n)$, is both the left- and the right-inverse of $B$. 
Let 
$\tilde T: Y \to Y$ be given by 
\begin{equation*}\label{eq:deftildeTT}
\tilde T := B^{-1} T_{k(-a)} B.
\end{equation*}
Further, let $\tilde k := (k_n(-a)b_n)$, which belongs to $Y^*$. Since $k_n(-a) = - a_n\pi_n/b_n$, we have that $\tilde k_n = - a_n\pi_n$.  
Thus, $\tilde T$ can be expressed as follows: $$\tilde T = A + \bfo \tilde k.$$

By the fact that $k(-a)\in \ell^1$ and by Theorem~\ref{thm:main1}, we have that 
$$\Sigma(T_{k(-a)}) = \{-a_n \mid n\in \N_0\}, \quad \mbox{where } a_0 := 0.$$ 
Using the same arguments as in the proof of Theorem~\ref{thm:main1} (more specifically, the arguments toward the end of Subsection~\S\ref{ssec:proofofmain1}), we have that
$\Sigma(\tilde T) = \Sigma(T_{k(-a)})$. It follows that
\begin{equation}\label{eq:spectrumoftildeT}
\Sigma(\tilde T) = \{-a_n \mid n\in \N_0\}.
\end{equation}

Now, we introduce the following linear dynamical system, termed the {\em $y$-system}, with $Y$ the state-space:
\begin{equation}\label{eq:systemfory}
    \dot y(t) = \tilde T y(t).
\end{equation}
Note that if $x(0)$ is the initial condition of~\eqref{eq:reproducek-a} and if $y(0) = B^{-1} x(0)$, then 
$$
y(t) = \exp(\tilde T t) y(0) = B^{-1} \exp(T_{k(-a)} t) B y(0) = B^{-1} x(t),
$$
where $x(t)$ is the solution of~\eqref{eq:reproducek-a}. It follows that 
\begin{equation*}\label{eq:ytxtnormequal}
\|y(t)\|_{Y} = \|x(t)\|_{X}.
\end{equation*}
Thus, system~\eqref{eq:reproducek-a} is stable (resp. asymptotically stable) if and only if the $y$-system is stable (resp. asymptotically stable).

\subsubsection{Diagonalizability of \texorpdfstring{$\tilde{T}$}{T}}\label{ssec:diagonalizability}
Let $Y$ be the Banach space given in~\eqref{eq:defXbXb}.   
We define a linear operator $P: Y \to Y$ as follows: 
\begin{equation*}\label{eq:defoperatorP}
P: (y_n) \mapsto \left ( \sum_{m = 1}^\infty \frac{a_m \pi_m y_m}{a_n + a_m}  \right ), 
\end{equation*}
where $\pi = (\pi_n)$ is given as in~\eqref{eq:defpi}. We will see soon that $P$ is well defined. 
In its matrix representation, $P$ can be viewed as a variation of an infinite-dimensional Cauchy matrix: 
\begin{equation*}\label{eq:matrixrep}
P = \left [ P_{ij}:=  \frac{a_j \pi_j}{a_i + a_j} \right ]_{1\leq i,j < \infty},
\end{equation*}
where $P_{ij}$ denotes the $ij$th entry of $P$. 

We establish below relevant properties of the operator $P$. 
Recall that the constant $\kappa > 0$ is introduced in~\eqref{eq:amdividebml1first}: We have shown that the columns and the rows of $\Phi$ are elements of~$\ell^1$ and their $\ell^1$-norms are bounded above by~$\kappa$. 
We have the following result:  

\begin{Proposition}\label{prop:pwelldefined}
    For any $y\in Y$, $Py \in Y$. Moreover, $P$ is bounded and satisfies 
    \begin{equation*}\label{eq:normforoperatorP}
        \|P\|_{\mathcal{B}(Y)} \leq \kappa \|\pi\|_{\ell^\infty} .
    \end{equation*}
\end{Proposition}

\begin{proof}
Let $y':= Py$. We show below that 
$$\|y'\|_Y  < \kappa \|\pi\|_{\ell^\infty} \|y\|_Y.$$ 
Consider two cases: (1) $X = \ell^\infty$, $c$, or $c_0$ and (2) $X = \ell^p$ for $1 \leq p < \infty$.   

\xc{Case 1: $X = \ell^\infty$, $c$, or $c_0$.}  In this case, we have that for each $i\in \N$,  
$$
|b_i y'_i|  = \sum_{j = 1}^\infty \frac{a_j b_i}{(a_j + a_i)b_j} |\pi_j|   |b_j y_j| 
 \leq \|\pi\|_{\ell^\infty} \|y\|_{Y} \sum_{j = 1}^\infty \phi_{ij} 
 \leq \kappa \|\pi\|_{\ell^\infty} \|y\|_{Y},
$$
where the last inequality follows from~\eqref{eq:amdividebml1first}. 
Thus, 
$$
\|y'\|_Y \leq \kappa \|\pi\|_{\ell^\infty} \|y\|_Y.
$$
This establishes the result for Case~1. 

\xc{Case 2: $X = \ell^p$ for $1\leq p < \infty$.} Let $1 < q \leq \infty$ be such that $1/p + 1/q = 1$. 
We have that
\begin{equation}\label{eq:holdinequalityprelim}
|b_i y'_i|  \leq \|\pi\|_{\ell^\infty} \sum_{j = 1}^\infty \phi_{ij} |b_j y_j|  \leq \|\pi\|_{\ell^\infty} \sum_{j = 1}^\infty \phi_{ij}^{1/q} \cdot \phi_{ij}^{1/p}|b_jy_j|
\end{equation}
For convenience, let $\varphi_i:= (\phi_{ij}^{1/q})_{j\in \N}$ and $\psi_i:= (\phi_{ij}^{1/p} |b_jy_j|)_{j\in \N}$. 
Note that 
\begin{equation}\label{eq:lqlpvarphi}
\|\varphi_i\|_{\ell^q} \leq \kappa^{\frac{1}{q}} \quad  \mbox{and} \quad \|\psi_i\|_{\ell^p}\leq \kappa^{\frac{1}{p}} \|y\|_Y, 
\end{equation}
and hence, $\varphi_i\in \ell^q$ and $\psi_i\in \ell^p$. We then proceed with~\eqref{eq:holdinequalityprelim} and obtain that
\begin{equation}\label{eq:boundonbiy'i}
|b_i y'_i| \leq \|\pi\|_{\ell^\infty} \|\varphi_i \psi_i\|_{\ell^1} \leq  
\|\pi\|_{\ell^\infty} \|\varphi_i\|_{\ell^q} \|\psi_i\|_{\ell^p} \leq   \kappa^{\frac{1}{q}} \|\pi\|_{\ell^\infty} \|\psi_i\|_{\ell^p},
\end{equation}
where the second inequality follows from the H\"older's inequality and the third inequality follows from~\eqref{eq:lqlpvarphi}. 
It then follows that
\begin{align*}
\|y'\|_Y^p & = \sum_{i = 1}^\infty |b_iy'_i|^p  \\
& \leq \kappa^{\frac{p}{q}} \|\pi\|_{\ell^\infty}^p \sum_{i = 1}^\infty \|\psi_i\|_{\ell^p}^p \\
& \leq \kappa^{\frac{p}{q}} \|\pi\|_{\ell^\infty}^p\sum_{i = 1}^\infty \sum_{j = 1}^\infty \phi_{ij} |b_jy_j|^p \\
& \leq \kappa^{\frac{p}{q}}\|\pi\|_{\ell^\infty}^p \sum_{j = 1}^\infty |b_jy_j|^p \sum_{i = 1}^\infty \phi_{ij} \\
& \leq \kappa^{\frac{p}{q} + 1} \|\pi\|_{\ell^\infty}^p \|y\|_Y^p,
\end{align*}
where we have used~\eqref{eq:boundonbiy'i} and~\eqref{eq:amdividebml1column} to establish the first and the last inequality, respectively. We conclude that
$$
\|y'\|_Y \leq  \kappa^{\frac{1}{q} + \frac{1}{p}} \|\pi\|_{\ell^\infty} \|y\|_Y = \kappa \|\pi\|_{\ell^\infty} \|y\|_Y.
$$
This completes the proof.
\end{proof}

We next show that the operator $P$ has an inverse, which is itself. 

\begin{Proposition}\label{prop:ppi}
    The operator $P$ satisfies $P^2 = I$.
\end{Proposition}

\begin{proof}
Let $P_j$ be the $j$th column of $P$, viewed as an element of $Y$, and $P^\top_i$ be the $i$th row of $P$, viewed as an element of $Y^*$. We show that 
\begin{equation}\label{eq:ppi}
P_i^\top P_j = \delta_{ij},
\end{equation}
where $\delta_{ij}$ is the Kronecker delta. To establish~\eqref{eq:ppi}, we will first construct $P_j(N)$ and $P^\top_i(N)$ for $N\in \N$ as approximations of $P_j$ and $P^\top_i$, respectively, next show that $P_i^\top(N) P_j(N) = \delta_{ij}$, and then prove that $P_j(N)$ and $P^\top_i(N)$ converge to $P_j$ and $P^\top_i$, respectively.

\xc{Construction of $P_j(N)$ and $P^\top_i(N)$.} For each $N\in \N$, we define an eventually zero sequence $\pi(N) = (\pi_n(N))$ as follows:
\begin{equation*}\label{eq:defpifinitecase}
\pi_n(N):= 
\begin{cases}
2\prod_{m = 1, m\neq n}^N\frac{1+ a_m/a_n}{1 - a_m/a_n} & \mbox{if } 1\leq n \leq N, \\
0 & \mbox{if } n \geq N + 1.
\end{cases}
\end{equation*}
Note that  $\pi(N)$ is {\em not} obtained by truncating $\pi$ since $\pi_n(N) \neq \pi_n$ for $1\leq n \leq N$.   
Let $P(N): Y\to Y$ be defined in the same way as $P$, but with $\pi_j$ replaced by $\pi_j(N)$, i.e.,  
$$
P(N):= \left [P_{ij}(N) := \frac{a_j\pi_j(N)}{a_i + a_j} \right ]_{1\leq i, j< \infty}. 
$$ 
Similarly, let $P^\top_i(N)$ and $P_j(N)$ be the $i$th row and the $j$th column of $P(N)$, respectively. Note, in particular, that 
$$P^\top_i(N) = (P_{i1}(N),\cdots, P_{i N}(N), 0, 0, \cdots ),$$ 
is an eventually zero sequence.    

\xc{Proof that $P^\top_i(N)P_j(N) = \delta_{ij}$.} 
Consider the following $N\times N$ matrix obtained by truncation of $P(N)$: 
$$
P'(N) := \left [ P_{ij}(N) \right ]_{1\leq i, j\leq N}. 
$$ 
It is known (see, e.g.,~\cite{schechter1959inversion}) that the inverse of a Cauchy matrix is given by 
$$
\left [ \frac{1}{a_i + a_j} \right ]_{1\leq i, j \leq N}^{-1} = \left [ \frac{a_ia_j\pi_i(N)\pi_j(N)}{a_i + a_j}\right ]_{1\leq i, j \leq N}, 
$$
and hence, 
\begin{equation}\label{eq:finitcaseppi}
P'^2(N) = I.
\end{equation}
By construction of $P_j(N)$ and $P_i^\top (N)$, we have that
$$
P^\top_i(N) P_j(N) = P'^\top_i(N)P'_j(N) = \delta_{ij},  
$$
where the last equality follows from~\eqref{eq:finitcaseppi}.

\xc{Convergence of $P_j(N)$.}  
Note that if $N \geq j$, then $\pi_j(N)$ has the same sign as $\pi_j$. Moreover, $|\pi_j(N)|$ is monotonically increasing in $N$, and it converges to $|\pi_j|$ as $N\to\infty$. Thus,  
\begin{equation*}\label{eq:convergenceofpin}
\lim_{N\to\infty} \pi_j(N) = \pi_j, \quad \mbox{for all } n \in \N.   
\end{equation*}
Now, for each $i\in \N$, we have that 
\begin{equation*}\label{eq:convergenceofPij}
|P_{ij}(N) - P_{ij}| = \frac{a_j}{a_i + a_j} |\pi_j(N) - \pi_j| <   |\pi_j(N) - \pi_j|. 
\end{equation*}
It follows that 
$$
\lim_{N\to\infty}\|P_j(N) - P_j\|_{Y} \leq  \|b\|_X \lim_{N\to\infty}|\pi_{j}(N) - \pi_{j}| = 0.
$$

\xc{Convergence of $P^\top_i(N)$.} 
Let $\tilde P_{i}^\top := (P_{ij}/{b_j})_{j \in \N}$. Note that $\tilde P_{i}^\top\in \ell^1$; indeed, 
$$
\|\tilde P_{i}^\top \|_{\ell^1} = \frac{1}{b_i}\sum_{j = 1}^\infty \phi_{ij}\pi_j \leq \frac{\kappa \|\pi\|_{\ell^\infty}}{b_i}.  
$$
Let $\tilde P_{ij}(N):= P_{ij}(N)/b_j$ and  $\tilde P_i^\top(N):= (\tilde P_{ij}(N))_{j \in \N}$. 
We show below that $\tilde P_i^\top(N)$ converges, in the $\ell^1$-sense, to $\tilde P_i^\top$. 
The arguments will be similar to those used in the proof of Proposition~\ref{prop:convergenceofkN}. 
Specifically, for any $\epsilon > 0$, there exist two positive integers $N'$ and $N_\epsilon$, with $N'\leq N_\epsilon$, such that
$$\sum_{j = N' + 1}^\infty |P_{ij}|/b_j < \epsilon / 2,$$ and, moreover, 
$$
0 < 1 - \frac{\pi_{j}(N)}{\pi_j} \leq \frac{\epsilon}{2\|\tilde P_i^\top \|_{\ell^1}} \quad \mbox{for all } j\leq N' \mbox{ and for all } N \geq N_\epsilon. 
$$
Since $\pi_j(N)$ and $\pi_j$ have the same sign for $N\geq j$ and since $|\pi_j(N)| \leq |\pi_j|$, we have that
$$
| \tilde P_{ij} - \tilde P_{ij}(N)| = \left [1 - \frac{\pi_{j}(N)}{\pi_j} \right ] \frac{|P_{ij}|}{b_j} \leq \frac{|P_{ij}|}{b_j} = |\tilde P_{ij}|, \quad \mbox{for } j\leq N. 
$$
The above arguments then imply that for any $N \geq N_\epsilon$,  
\begin{align}
\|\tilde P^\top_i - \tilde P_i^\top (N) \|_{\ell^1} & \leq \sum_{j = 1}^{N'} \left [1 - \frac{\pi_{j}(N)}{\pi_j} \right ] \frac{|P_{ij}|}{b_j}  + \sum_{j = N' + 1}^{\infty} \left [1 - \frac{\pi_{j}(N)}{\pi_j} \right ] \frac{|P_{ij}|}{b_j}  \notag\\
& \leq \frac{\epsilon}{2\|\tilde P^\top_i\|_{\ell^1}} \sum_{j = 1}^{N'} \frac{|P_{ij}|}{b_j} +  \epsilon/2 < \epsilon/2 + \epsilon/2 = \epsilon. \notag
\end{align}
Thus, $\lim_{N\to\infty} \|\tilde P^\top_i(N) - \tilde P_i^\top\|_{\ell^1}  = 0$  
and hence, 
$$
\lim_{N\to\infty} \|P^\top_i(N) -  P_i^\top\|_{Y^*}= 0. 
$$
We conclude that
$$
P^\top_i P_j = \lim_{N\to\infty} P^\top_i(N) P_j(N) = \delta_{ij}.
$$
This completes the proof. 
\end{proof}

Recall from~\eqref{eq:spectrumoftildeT} that $\Sigma_{\rm disc}(\tilde T) = \{-a_n \mid n\in \N\}$. Since the $a_n$'s are pairwise distinct, the algebraic/geometric multiplicity of each $-a_n$ is~$1$. 
Let $\tilde T^*: Y^*\to Y^*$ be the adjoint of~$T$. 
The following result characterizes the eigenvectors of $T$ and of $T^*$:   

\begin{Proposition}\label{prop:eigenvectorsofP}
Let $P_j$ and $P_i^\top$ be the $j$th column and the $i$th row of $P$, viewed as elements of $Y$ and of $Y^*$, respectively. Then, the following hold:
\begin{enumerate}
    \item Each $P_j$ is an eigenvector of $\tilde T$ corresponding to $-a_j$.
    \item Each $P_i^\top$ is an eigenvector of $\tilde T^*$ corresponding to $-a_i$. 
\end{enumerate}
\end{Proposition}

\begin{proof} We establish below the two items.

\xc{Proof of item~1.} 
Let $v_j  \in Y$ be an eigenvector of $\tilde T = A + \bfo \tilde k$ corresponding to $-a_j$. 
Let $v_{ij}$ be the $i$th entry of $v_j$. 
Then, entry-wise, $v_j$ satisfies the following: 
\begin{equation}\label{eq:eigenvectorPj}
(a_i + a_j) v_{ij} + \tilde k v_j = 0, \quad \mbox{for all } i\in \N, 
\end{equation}
which implies that $v_j \propto P_j$. This establishes item~1. Note that if we replace $v_j$ with $P_j$ in~\eqref{eq:eigenvectorPj}, then we obtain
$$
0 = (a_i + a_j) P_{ij} + \tilde k P_j =  a_j \pi_j \left [ 1 - \sum_{i = 1}^\infty \frac{a_i\pi_i}{a_i + a_j} \right ] = 0,
$$
and hence,
\begin{equation}\label{eq:pibfoisone}
\sum_{i = 1}^\infty \frac{a_i\pi_i}{a_i + a_j} = 1, \quad \mbox{for all } j \in \N.
\end{equation}

\xc{Proof of item~2.} Note that for the case $X = \ell^\infty$, elements of $Y^*$ may not be represented by sequences. Thus, for this item, we verify  by computation that 
$\tilde T^* P_i^\top = - a_i P_i^\top$. 
By~\eqref{eq:pibfoisone}, we have that $P^\top_i \bfo = 1$ for all $i\in \N$. 
Thus,   
$$
\tilde T^* P_i^\top  = P^\top_i A + (P^\top_i\bfo) \tilde k   = \left ( \frac{a_j^2 \pi_j}{a_i + a_j} \right )_{j\in \N} - \left ( a_j \pi_j \right )_{j\in \N}  = - a_i \left ( \frac{a_j \pi_j}{a_i + a_j} \right )_{j\in \N} = -a_i P^\top_i.
$$
This completes the proof. 
\end{proof}

The above results then lead to the following:

\begin{Corollary}\label{cor:diagonalizable}
We have that $P \tilde T P^{-1} = - A$.
\end{Corollary}

\begin{proof}
This corollary follows directly from Propositions~\ref{prop:ppi} and~\ref{prop:eigenvectorsofP}. 
\end{proof}

\subsubsection{Proof of the two items of Theorem~\ref{thm:sufficientnew}}\label{ssec:stability}
Let $z(t):= Py(t) \in Y$. By~\eqref{eq:systemfory} and by Corollary~\ref{cor:diagonalizable}, we have that
\begin{equation}\label{eq:thezsystem}
\dot z(t) = -A z(t)
\end{equation} 
We call~\eqref{eq:thezsystem} the {\em $z$-system}, with $Y$ the state space. 
By Proposition~\ref{prop:ppi}, $P^2 = I$ and hence,  
$y(t) = P z(t)$. By Proposition~\ref{prop:pwelldefined}, $P$ is a bounded operator. Thus, 
$$
\|y(t)\|_Y \leq \|P\|_{\mathcal{B}(Y)} \|z(t)\|_Y \quad \mbox{and} \quad 
\|z(t)\|_Y \leq \|P\|_{\mathcal{B}(Y)} \|y(t)\|_Y.
$$
Consequently, the original system~\eqref{eq:reproducek-a}, the $y$-system, and the $z$-system all share the same stability properties. 
In the sequel, we focus on the $z$-system.

\xc{Proof that the $z$-system is stable.}  
From~\eqref{eq:thezsystem}, we have that 
\begin{equation}\label{eq:thesolutionforzsystem}
z_n(t) = e^{-a_n t} z_n(0), \quad \mbox{for all } n\in \N, 
\end{equation}
which implies that 
\begin{equation*}\label{eq:boundforzt0}
\|z(t)\|_{Y} \leq \|z(0)\|_{Y}, \quad \mbox{for all } t\geq 0.
\end{equation*}
This establishes stability of the $z$-system~\eqref{eq:thezsystem}. 

\xc{Proof of item~1 of Theorem~\ref{thm:sufficientnew}.} We will now assume that 
\begin{equation}\label{eq:Xisconvergent}
X = c_0 \quad \mbox{or} \quad X = \ell^p, \quad \mbox{for } 1\leq p < \infty,
\end{equation}
and show that the $z$-system is asymptotically stable.  
We do so by showing that for any $\epsilon > 0$, there exists a time $T_{\epsilon}$ such that 
$\|z(t)\|_Y \leq \epsilon$ for all $t \geq T_{\epsilon}$. 
Let $z(t; N)$ be an eventually zero sequence defined as follows: 
$$
z_n(t;N) := 
\begin{cases}
z_n(t) & \mbox{if } 1\leq n \leq N, \\
0 & \mbox{otherwise}.  
\end{cases}
$$
By~\eqref{eq:Xisconvergent}, we have that 
$\lim_{N\to\infty}\|z(0; N) - z(0)\|_Y = 0$. 
Thus, by~\eqref{eq:thesolutionforzsystem}, there exists an $N_\epsilon$ such that 
\begin{equation}\label{eq:ztlessthanepsilon2}
\|z(t; N_\epsilon) - z(t)\|_Y < \epsilon/2, \quad \mbox{for all } t \geq 0.
\end{equation}
Now, let $\tau_{\epsilon}\geq 0$ be such that 
\begin{equation}\label{eq:definetepsilon}
\exp(-a_{N_\epsilon} \tau_{\epsilon}) \|z(0)\|_Y < \frac{\epsilon}{2 }.
\end{equation}
Since $(a_n)$ is monotonically decreasing, for fixed $t \geq 0$ the sequence $(\exp(- a_n t))$ is monotonically increasing. Thus, for $t \geq \tau_{\epsilon}$, 
\begin{align}\label{eq:truncatezforNepsilon}
\|z(t; N_\epsilon)\|_Y & \leq \exp(-a_{N_\epsilon} t) \|z(0; N_\epsilon)\|_Y  \notag \\ 
& \leq \exp(-a_{N_\epsilon} \tau_{\epsilon}) \|z(0; N_\epsilon)\|_Y \notag \\
& \leq \exp(-a_{N_\epsilon} \tau_{\epsilon}) \|z(0)\|_Y < \epsilon/2,
\end{align}
where the last inequality follows from~\eqref{eq:definetepsilon}.  
Combining~\eqref{eq:ztlessthanepsilon2} and~\eqref{eq:truncatezforNepsilon}, we obtain that for $t \geq \tau_{\epsilon}$, 
$$
\|z(t)\|_Y   \leq \|z(t; N_\epsilon)\|_Y + \|z(t) - z(t; N_\epsilon)\|_Y  < \epsilon/2 + \epsilon/2 = \epsilon.
$$
This establishes item~1 of Theorem~\ref{thm:sufficientnew}.  

\xc{Proof of item~2 of Theorem~\ref{thm:sufficientnew}.} We will now assume that $X = \ell^\infty$ or $X = c$, and show that the $z$-system is not asymptotically stable. Let $z(0):= (1/b_n)$. Note that $Bz(0) = \bfo \in X$, so $z(0)\in Y$. By~\eqref{eq:thesolutionforzsystem}, we have that
$$
b_n z_n(t) = e^{-a_nt}. 
$$
For each $n\in \N$, the minimum time $T_n$ such that $b_n z_n(t) \leq \epsilon < 1$ is given by
$T_n = - \ln \epsilon /a_n$. But then, $(T_n)$ is not bounded since $(a_n)$ converges to $0$.  

This completes the proof of Theorem~\ref{thm:sufficientnew}. \hfill{\qed}

\subsection{Proof of Theorem~\ref{thm:sufficient}}
For ease of presentation, we again assume that all the $b_n$'s are positive. We reproduce below the hypothesis of the theorem:
\begin{equation}\label{eq:condforabstab}
 a_{n+1}/a_n < \nu_0 \quad \mbox{and} \quad  \nu_1 < b_{n+1}/b_n < \nu_2, \quad \mbox{for all } n \in \N,  
\end{equation}
where $0 < \nu_0 < \nu_1 < \nu_2 < 1$. 
Under this hypothesis, we will first show that $\pi \in \ell^\infty$, and then show that $\Phi$ spatially exponentially decays. 

\subsubsection{Proof that $\pi\in \ell^\infty$}\label{sssec:pilinfinity}
We show that $(\ln|\pi_n|)\in \ell^\infty$. From the definition~\eqref{eq:defpi} of $\pi_n$,  
$$
\ln |\pi_n| = \ln 2 + \sum_{m = 1, m\neq n}^{\infty}\ln \left |\frac{a_m/a_n + 1}{a_m/a_n - 1}\right |.
$$
We provide below upper bounds for the addends in the above summation. Consider two cases: $m < n$ and $m > n$.   

\xc{Case 1: $m < n$.} Note that $a_m/a_n \geq \nu_0^{m-n} > 1$ and hence,  
$$
    1 < \frac{a_m/a_n + 1}{a_m/a_n - 1} \leq \frac{\nu_0^{m - n} + 1}{\nu_0^{m - n} - 1} = 1 + \frac{2\nu_0^{n - m}}{1 - \nu_0^{n - m}} \leq 1 + \frac{2\nu_0^{n - m}}{1 - \nu_0}.
$$
It follows that
$$
    \sum_{m = 1}^{n-1}\ln \frac{a_m/a_n + 1}{a_m/a_n - 1} \leq \sum_{m = 1}^{n-1} \ln \left [ 1 + \frac{2\nu_0^{n - m}}{1 - \nu_0} \right ] < \frac{2}{1 - \nu_0}\sum_{m = 1}^{n-1} \nu_0^{n - m} < \frac{2\nu_0}{(1 - \nu_0)^2},
$$
where the second inequality follows from the fact that $\ln(1 + x) < x$, for $x > 0$.

\xc{Case 2: $m > n$.} In this case, we have that $a_m/a_n \leq \nu_0^{m-n} < 1$ and  
$$
    1 < \frac{1 + a_m/a_n}{1 - a_m/a_n} \leq \frac{1 + \nu_0^{m - n} }{1 - \nu_0^{m - n}} < 1 + \frac{2\nu_0^{m - n}}{1-\nu_0}. 
$$
Thus,
\begin{align*}
    \sum_{m = n+1}^{\infty}\ln \frac{1 + a_m/a_n}{1 - a_m/a_n} & \leq \sum_{m = n + 1}^{\infty} \ln \left [ 1 + \frac{2\nu_0^{m - n}}{1 - \nu_0} \right ] \\
    & < \frac{2}{1 - \nu_0}\sum_{m = n + 1}^\infty \nu_0^{m - n} = \frac{2\nu_0}{(1 - \nu_0)^2}.
\end{align*}

Combining the arguments for the two cases, we conclude that for all $n\in \N$,
\begin{equation*}\label{eq:boundforpi}
     \ln |\pi_n| = \ln 2 + \sum_{m = 1}^{n-1}\ln \frac{a_m/a_n + 1}{a_m/a_n - 1}  + \sum_{m = n+1}^{\infty}\ln \frac{1 + a_m/a_n}{1 - a_m/a_n} < \ln 2 + \frac{4\nu_0}{(1 - \nu_0)^2}.
\end{equation*}
This completes the proof. \hfill{\qed}

\subsubsection{Proof that $\Phi$ spatially exponentially decays}\label{sssec:spatialexponentialdecay}
Let $\mu := \max\{\nu_0/\nu_1, \nu_2\}$. By~\eqref{eq:condforabstab}, we have that $\mu < 1$. We show below that
\begin{equation}\label{eq:decayboundforPhi}
\phi_{ij} \leq \mu^{|i - j|}, \quad \mbox{for all } i,j\in \N.
\end{equation}
Note that if $i = j$, then $\phi_{ii} = 1/2 < 1$, so~\eqref{eq:decayboundforPhi} is satisfied. For $i\neq j$,  we consider  two cases: $i< j$ and $i > j$.

\xc{Case 1: $i < j$.} We have that
$$
\phi_{ij} = \frac{b_i/b_j}{1 + a_i/a_j} < \frac{b_i/b_j}{a_i/a_j} \leq \left (\frac{\nu_0}{\nu_1}\right )^{j - i} \leq \mu^{j - i}.
$$

\xc{Case 2: $i > j$.} In this case, 
$$
\phi_{ij} = \frac{b_i/b_j}{1 + a_i/a_j} < b_i/b_j \leq \nu_2^{i - j} \leq \mu^{i - j}.
$$
This completes the proof. \hfill{\qed}

\section{Conclusion}
In this paper, we have addressed the problem of pole placement and the problem of feedback stabilization for the following discrete linear ensemble system: $\dot x_n(t) = a_n x_n(t) + b_n u(t)$, where the state space $X$ is either $c$, or, $c_0$, or, $\ell^p$ for $1 \leq p \leq \infty$. 
We have established several necessary or sufficient conditions on the sequences $(a_n)$ and $(b_n)$ for the existence of a feedback gain $k\in X^*$ such that the operator $T_k = (A + bk): X\to X$ has its spectra in the left half plan and for the feedback system $\dot x(t) = T_k x(t)$ to be (asymptotically) stable. These conditions have been formulated as theorems in Section~\S\ref{sec:mainresult} and their proofs have been presented in Section~\S\ref{sec:proofs}.   
These two problems have natural connections with rank-one  perturbations of (compact) operators and with stability of uniformly continuous semigroups. The main results of this paper may be of independent interest in operator theory as well.   
\printbibliography
\end{document}